\documentclass[11pt,final]{amsart}

\usepackage{lmodern}
\usepackage[utf8]{inputenc}
\usepackage[english]{babel}
\usepackage[T1]{fontenc}

\usepackage[left=2.05cm,right=2.05cm,top=2.95cm,bottom=2.8cm]{geometry}
\usepackage[foot]{amsaddr} % Pour mettre l'auteur en dessous

\usepackage[active]{srcltx}
\usepackage{graphicx}
\usepackage[numbers]{natbib}
\usepackage{array}
\usepackage{eqnarray}
\usepackage{tikz}
\usepackage{amsmath}
\usepackage{amssymb} 
\usepackage{amsthm}
\usepackage{mathabx}
\usepackage{bbm} % Pour faire \mathbb{1}
\usepackage{bm} % Pour faire des lettres grecques en italique
\usepackage{esint}
\usepackage{enumitem}
\usepackage{makecell}

\usepackage{color}
\usepackage[pdfborder={0 0 0}]{hyperref}

\numberwithin{equation}{section} % Numérotation des équations

\newcommand{\R}{\mathbb{R}}
\newcommand{\N}{\mathbb{N}}

\newcommand{\Leb}{\mathcal{L}} % Lebesgue measure
\newcommand{\M}{\mathcal{M}} % Measures
\newcommand{\ddr}{\mathrm{d}} % d droit pour les intégrales
\newcommand{\dr}{\partial}

\newcommand{\Phiinc}{\Phi_\text{inc}}
\newcommand{\W}{\mathcal{W}}

\newcommand{\nO}{\mathbf{n}_{\Omega}} % Normale extérieure à Omega

\newcommand{\1}{\mathbbm{1}} 

\newcommand{\dst}[1]{\displaystyle{#1}}

\newtheorem{thm}{Theorem}[section]
\newtheorem{prop}[thm]{Proposition}
\newtheorem{lem}[thm]{Lemma}
\newtheorem{cor}[thm]{Corollary}

\theoremstyle{definition}
\newtheorem{definition}[thm]{Definition}
\newtheorem{example}[thm]{Example}

\theoremstyle{remark}
\newtheorem{remark}[thm]{Remark}

\begin{document}

\title{Hidden convexity in a problem of nonlinear elasticity} 

\date{\today}
\author{Nassif Ghoussoub}
\address{Department of Mathematics, University of British Columbia}
\email{\{nassif,yhkim,lavenant,azp\}@math.ubc.ca}
\author{Young-Heon Kim}
\author{Hugo Lavenant}
\author{Aaron Zeff Palmer}

\begin{abstract}
We study compressible and incompressible nonlinear elasticity variational problems in a general context. 
Our main result gives a sufficient condition for an equilibrium to be a global energy minimizer, in terms of convexity properties of the pressure in the deformed configuration. We also provide a convex relaxation of the problem together with its dual formulation, based on measure-valued mappings, which coincides with the original problem under our condition. 
\end{abstract}

\maketitle

\tableofcontents

\section{Introduction}

In this article, we study a problem of calculus of variations of the form
\begin{equation}
\label{eq_informal}
\min_u \left\{ \int_\Omega \Big(W(\nabla u) +F(u)\Big)\ddr \Leb_\Omega + \Phi( u \# \Leb_\Omega ); \ u : \Omega \to D \text{ and } u = g \text{ on } \dr \Omega \right\}
\end{equation} 
where the unknown, $u : \Omega \subset \R^d \rightarrow  D\subset \R^k$, is a vector-valued function representing the deformation of an elastic solid when $d=k$.  The data consists of the boundary values fixed by a function $g : \dr \Omega \to \dr D$; a potential energy $F : \Omega \times D\to \R$; and the hyper-elastic stored energy $W:\Omega \times \R^{k \times d} \to \R$ (dependence on $x\in \Omega$ is suppressed in our notation for \eqref{eq_informal} and what follows). The functional $\Phi$ is convex on the set of finite nonnegative measures over $D$. We let $\Leb_\Omega$ denote the Lebesgue measure restricted to $\Omega$  and $u \# \Leb_\Omega$ is the push-forward image measure of $\Leb_\Omega$ by $u$. 

The first example is the incompressibility constraint, imposed by
$$
	\Phi_{\text{inc}}(\mu) = \begin{cases} 0 & \mu = \Leb_D \\ +\infty & {\rm otherwise}.\end{cases}
$$
When $d=k$, any $u$ with $u \# \Leb_\Omega=\Leb_D$ will be called incompressible, and if such a $u$ is sufficiently differentiable and injective then equivalently the Jacobian equation $\det(\nabla u)=1$ holds on $\Omega$. We will also consider integral function of the density, that is
\begin{equation}
\label{eq_phi_integral}
	\Phi(\mu) = \begin{cases} \dst{ \int_D  \phi \left( \frac{\ddr \mu}{\ddr \Leb_D} \right) \ddr \Leb_D} & \text{if } \mu \ll \Leb_D \\ 
	+\infty & {\rm otherwise}.\end{cases}
\end{equation}
where $\phi : [0, + \infty) \to \R$ is a convex function.
As shown in Section \ref{subsec_elasticity}, this will correspond to the energy of \emph{compressible} deformations in nonlinear elasticity.

The presence of $\Phi$ composed with the image measure of $u$ makes problem \eqref{eq_informal} nonlinear and nonconvex. However, \emph{we will provide a sufficient condition for a solution of the Euler-Lagrange equation of \eqref{eq_informal} to be the unique global minimizer of \eqref{eq_informal}}. Moreover, we will see that under this assumption the problem in fact coincides with a convex relaxation built on measure valued mappings. 

\bigskip

First, let us motivate this study. Though we will mainly deal with examples where $d=k$, that is $\Omega$ and $D$ lie in the same Euclidean space, our results also apply to $d \neq k$. In particular, the case $k=1$ shares some connections with hydrodynamics \cite{Brenier2015}. Problem \eqref{eq_informal} appears at least in two contexts.

\subsection{Optimal transport} 

If $\Omega = D$ and $W = 0$ (and we omit the boundary conditions) while $\Phi = \Phi_\text{inc}$ and $F(x,u) = -f(x) \cdot u$, then the problem reads
\begin{equation*}
\max_u \left \{ \int_\Omega f\cdot u\, \ddr\Leb_\Omega; \ u \# \Leb_\Omega = \Leb_\Omega \right\}, 
\end{equation*}
and one recovers the problem of polar factorization studied by Brenier \cite{Brenier1987, Brenier1991}. In this case, provided $f \# \Leb_\Omega$ does not charge small sets, then $f$ can be uniquely written $f = (\nabla \omega ) \circ u$ where $\omega$ is convex and $u$ is a solution of the problem. Actually, $\nabla \omega$ is the unique gradient of a convex function mapping $f \# \Leb_\Omega$ onto $\Leb_\Omega$. 

Hence Problem \eqref{eq_informal} can be seen as an optimal transport problem to which one adds a gradient penalization. Such problems have been considered in the PhD thesis of Louet \cite{LouetPhD} which in particular studies in great details the existence of a solution. What we have here can be seen as a simplified version of their problem: we only look at $\Phi(u \# \Leb_\Omega)$ while they consider $\Phi(u \# \alpha)$ with $\alpha \in \M(\Omega)$ possibly singular. However, the result that we have (solutions of the Euler-Lagrange equations are global minimizers under appropriate conditions) was not addressed at all in their work.

\subsection{Non Linear elasticity}
\label{subsec_elasticity}

In nonlinear elasticity theory, one is led to consider problem of the form
\begin{equation}\label{eqn:elasticity}
\min_{u} \left\{ \int_\Omega\Big( W(\nabla u) +h(\det \nabla u)+F(u)\Big)\ddr \Leb_\Omega; \  u=g\ \mbox{on }\dr\Omega \right\},
\end{equation}
where $k=d$ and $u : \Omega \to D$ is the deformation of an elastic solid. The functions $W$ and $h$ are assumed to be convex, in such a way that $C \mapsto W(C) + h(\det C)$ is a polyconvex function on the set of matrices. The function $h(t)$ usually tends to $ + \infty$ when $t \to 0$, and a limit case is when $h(t) = + \infty$ for all $t$ but $t=1$, which implies that $u \# \Leb_\Omega = \Leb_D$ and $u$ is incompressible.

By a change of variable formula, this problem falls into our framework. Let us define the function $\phi_h$ by $\phi_h(s) = h(s^{-1}) s$ and the functional 
\begin{equation*}
	\Phi_h(\mu) = \begin{cases} \dst{ \int_D  \phi_h \left( \frac{\ddr \mu}{\ddr \Leb_D} \right) \ddr \Leb_D} & \text{if } \mu \ll \Leb_D \\ 
	+\infty & {\rm otherwise}.\end{cases}
\end{equation*}
which we already introduced in \eqref{eq_phi_integral}. Then, for injective $u$ we can write
\begin{equation*}
\int_\Omega h(\det \nabla u)\ddr \Leb_\Omega = \Phi_h(u \# \Leb_\Omega).
\end{equation*}
To justify this, we can note that $\frac{\ddr \Leb_D}{\ddr u\# \Leb_\Omega}\circ u = \det(\nabla u)$ and calculate
$$
	\Phi_h(u \# \Leb_\Omega)=\int_D \phi_h \left( \frac{\ddr u\# \Leb_\Omega}{\ddr \Leb_D} \right) \ddr \Leb_D=\int_D h \left( \frac{\ddr \Leb_D}{\ddr u\# \Leb_\Omega} \right) \ddr u\#\Leb_\Omega=\int_\Omega h (\det \nabla u) \, \ddr\Leb_\Omega.
$$
The function $\phi_h$ is convex (hence the functional $\Phi_h$ is also convex on the set of measures) provided $h$ is convex, since we can write it as a supremum of convex functions:
\begin{equation*}
\phi_h(s) = \sup_r ( r - s h^*(r) )
\end{equation*} 
(with $h^*$ being the Legendre transform of $h$). 
The now classical work of Ball \cite{Ball1976convexity}, guarantees conditions on $W$ and $h$ such that minimizers exist in a class of weakly differentiable Sobolev functions. Further regularity of these minimizers remains unknown as a major unsolved problem.

One may ask why in (\ref{eq_informal}) we have prescribed the codomain, $D$, along with the boundary conditions, whereas in the problem of elasticity, (\ref{eqn:elasticity}), one only prescribes the boundary conditions.  For the sufficiently regulary nonsingular deformations of elasticity, $f=g$ on $\partial\Omega$ implies $f(\Omega)=g(\Omega)=D$; see Theorem 5.5-2 of \cite{ciarlet1988mathematical}.
However, in the sequel we work with (\ref{eq_informal}) and we allow for less regular and more singular maps, thus we prescribe the codomain as an additional constraint.

The question of uniqueness of solutions to the equilibrium equations, as well as whether they are global minimizers is an open question in calculus of variations \cite{Ball2002, Ball2010progress}, with counterexamples for some special cases. Examples inspired from the theory of elasticity will be discussed in Section \ref{section_examples}. We can already highlight some insight from this theory:
\begin{itemize}
\item To get uniqueness of equilibrium, one needs to restrict to pure displacement on the boundaries, that is, only Dirichlet boundary conditions as we have done here.
\item There are situations where allowing discontinuous deformations can result in multiple equilibria that all are global minimizers.  Superlinear growth of $C\mapsto W(C)$ or stronger  seems necessary to avoid the phenomenon of \emph{cavitation} and guarantee uniqueness. 
\item If the domain $\Omega$ is not simply connected, it may be possible to construct examples with infinitely many local minima (hence solutions of the equilibrium equations) but only one global minimizer \cite{Post1997homotopy}.  
\end{itemize}
Our main result is that any solution of the equilibrium equations which is smooth and small in a certain sense (namely the pressure, expressed in the deformed configuration must be $\lambda$-convex with $\lambda$ not too negative) is the unique global minimizer (but there may be other local minima). It can be seen as a partial answer to \cite[Problem 8]{Ball2002}. %Our result will be the strongest when $W$ is quadratic, which makes it a little bit too restrictive for nonlinear elasticity.   

\bigskip

In the rest of this article, we first state rigorously the problem and derive the Euler-Lagrange equations. Then we state our sufficiency condition for being a global minimizer and we prove it. We provide examples to discuss the sharpness of the result. Eventually we introduce a convex relaxation based on measure valued mappings which coincides with the original problem under the same smallness assumption on the pressure. This last part does not provide new results on the original problem but we think is interesting on its own and relates to the parallel work of \cite{AwiGangbo2014}, which was the starting point of our study. 

\section{Existence of minimizers and optimality conditions}

\subsection{Notation and  the variational problem of interest}

Let $\Omega\subset \R^d$ and $D\subset \R^k$ be open and bounded domains with Lipschitz boundaries. We denote by respectively $\Leb_\Omega$ and $\Leb_D$ the Lebesgue measure restricted to $\Omega$ and $D$. 
%We denote by $\sigma$ the surface measure on $\dr \Omega$. 
The set of finite Radon measures on a metric space $X$ is denoted by $\M(X)$ and we endow it with the topology of weak* convergence. The set of positive Radon measures on $X$ is denoted by $\M_+(X)$. 

If $T : X \to Y$ and $\alpha \in \M(X)$ we denote by $T \# \alpha\in \M(Y)$ the push-forward of the measure $\alpha$ by $T$, that is the measure defined by $T \# \alpha(B) = \alpha(T^{-1}(B))$ for any Borel set $B$ in $Y$. We often make use of the change of variables, that if $T:X\to Y$ is $\alpha$ measureable and $g$ is continuous on $Y$ then
$$
	\int_X g\circ T\, \ddr \alpha = \int_Y g\, \ddr T\#\alpha.
$$

The notation $| x |$ stands for the Euclidean norm of $x$ if $x$ is a vector, and the Frobenius norm (also known as the Hilbert-Schmidt norm) if $x$ is a matrix. 

For the data of the problem, let $W : \Omega \times\R^{k \times d} \to \R$ be a Carathéodory function satisfying the coercivity and boundedness assumptions
\begin{equation*}
\frac{1}{C} ( |H|^{p} - 1 ) \leqslant W(H) \leqslant C (|H|^p+1)
\end{equation*}
for some $C > 0$ and $p>1$, and we assume that $H\mapsto W(H)$ is convex almost everywhere on $\Omega$.  
We take $g : \dr \Omega \to \partial D$ for the boundary condition, which we assume extends to a function on $\Omega$ with $g(\Omega)=D$, and $F$ a Carathéodory function satisfying for all $u : \Omega \to D$,  %$u\in \bar{D}$,
$$
	\int_\Omega |F(u)|\ddr \Leb_\Omega <+\infty
$$
and $u\mapsto F(u)$ is convex.

\begin{remark}
Not to overburden notations, we suppress the dependence on $x$ of $W$ and $F$ in the present article. For instance, an expression like $W(\nabla u)$ means $W(x, \nabla u(x))$ while $F(u) = F(x,u(x))$. 
\end{remark}

We denote by $W^{1,p}_{g}(\Omega,D)$ the set of Sobolev functions $u\in W^{1,p}(\Omega,\R^k)$ such that the trace of $u$ on $\dr \Omega$ is $g$ and the range of $u$ remains in $D$. %(We note that the injectivity constraint can be dropped, however, this results in a possibly distinct problem.)
We let $\bar D = D \cup \partial D$ and we take $\Phi : \M_+(\bar{D}) \to \R \cup \{ + \infty \}$ to be convex, bounded below and lower semi-continuous.  %We suppose that $\Phi(\mu)=+\infty$ whenever $\mu \notin \M_+(\bar{D})$.%, and that $\lim_{A\rightarrow \infty}\inf\{\Phi(\mu);\ \|\mu\|\geq A\}=+\infty$.

\begin{definition}
\label{def_energy}
We define the energy $E : W^{1,p}(\Omega,\R^k) \to \R \cup \{ + \infty \}$ as 
\begin{equation*}
E(u) = \int_\Omega \Big(W(\nabla u) +F(u)\Big)\ddr \Leb_\Omega + \Phi(u \# \Leb_\Omega). 
\end{equation*}
The problem with penalization of the image measure we are looking at is 
\begin{equation}
\label{eq_problem_main}
\inf \left\{ E(u); \ u \in W^{1,p}_{g}(\Omega, D) \right\}.
\end{equation}
\end{definition}
As noted in Proposition \ref{prop_existence_min} below, it is also possible to include the additional constraint that $u$ is injective on $\Omega$.

Note that it is not obvious that there exists at least one admissible competitor in \eqref{eq_problem_main}, or more generally that there exists a smooth map $u : \Omega \to D$ such that $u \# \alpha = \beta$, being $\alpha$ and $\beta$ two measures on $\Omega$ and $D$ respectively. If $\Omega = D$ is connected and has a smooth boundary, and $\alpha$ and $\beta$ have smooth densities bounded from below and above with respect to the Lebesgue measure, one can use \cite[Theorem 1.1]{dacorogna1990partial}. In such a case, Dirichlet boundary conditions can be prescribed. The result is easy to extend to $\Omega \neq D$ if there exists a smooth diffeomorphism between $\Omega$ and $D$ which has $g$ as boundary value. If $\Omega$ and $D$ are convex with a smooth boundary one can take instead $u$ to be the optimal transport map for the quadratic cost. This map $u$ is smooth \cite{caffarelli92, caffarelli96, urbas97} under slightly weaker smoothness assumptions on $\Omega, D,\alpha$ and $\beta$ compared to \cite[Theorem 1.1]{dacorogna1990partial}, but then Dirichlet boundary conditions cannot prescribed.
In the compressible case of $\Phi_h$, existence is easier to get, as we just need a smooth $u : \Omega \to D$ with prescribed boundary conditions and $\det \nabla u$ bounded from below and above. 
In any case, once existence of a competitor is known then existence of a minimizer of $E$ follows easily. 

\begin{prop}
\label{prop_existence_min}
Assume for $p >1$ that there exists $u\in W_{g}^{1,p}(\Omega,D)$ such that $\Phi(u\#\Leb_\Omega) < + \infty$. Then there exists a global minimizer of $E$ in $W_{g}^{1,p}(\Omega,D)$.  If $p>d$, then the result holds with the additional constraint that $u$ is injective.
\end{prop}

\begin{proof}
It comes from the direct method of calculus of variations.   As $\Phi$ is bounded from below, with the coercivity of $W$ it is easy to show that any minimizing sequence $(u_n)_{n \in \N}$ is bounded in $W^{1,p}(\Omega,\R^k)$. Up to extraction of a subsequence, $(u_n)_{n \in \N}$ converges to $u$ weakly in $W^{1,p}(\Omega, \R^k)$, strongly in $L^p(\Omega,\R^k)$ and almost everywhere. The latter convergence implies that $\lim_n u_n \# \Leb_\Omega = u \# \Leb_\Omega$ in the weak* topology.  Then all the terms of $E$ are lower semi-continuous so the limit $u$ is a global minimizer. In the case that we include the constraint that $u$ is injective, weak compactness still holds for injective functions in $W^{1,p}_g(\Omega,D)$ when $p>d$ using the result of \cite{ciarlet1987injectivity}. 
\end{proof}

We now wish to further compare our result with the classical results of nonlinear elasticity.  Indeed, the problem \eqref{eq_problem_main} is equivalent to the nonlinear elasticity problem \eqref{eqn:elasticity} when $p>d$ and $\Phi$ has the form of $\Phi_{\text{inc}}$ or $\Phi_h$.  We consider a convex, lower semicontinuous function $h:\R\rightarrow \R\cup \{+\infty\}$ such $\lim_{t\rightarrow 0} h(t)=+\infty$ and $h$ is bounded from below, for which $\Phi_h$ is convex, lower semicontinuous and bounded below.

\begin{lem}
We suppose that $g$ is smooth, injective and orientation preserving.  Let $u\in C^1(\bar{\Omega},\R^k)$ with $u=g$ on $\partial\Omega$ and $\det(\nabla u)>0$ on $\Omega$. Then $\Phi_{\text{inc}}(u \# \Leb_\Omega)<+\infty$ if and only if $\det \nabla u =1$ everywhere on $\Omega$.  Similarly, there holds
\begin{equation*}
\int_\Omega h (\det \nabla u)\, \ddr \Leb_\Omega = \Phi_h(u \# \Leb_\Omega).
\end{equation*}
\end{lem}

\begin{proof}
If $u\in C^1(\bar{\Omega},\R^k)$ with $u=g$ on $\partial \Omega$ with $\det(\nabla u)>0$, then $u$ is injective and $u(\Omega)=D$ by Theorem 5.5-2 in \cite{ciarlet1988mathematical}.  The pushforward measure is given by $\frac{\ddr u\# \Leb_\Omega}{\ddr \Leb_D}  = |\det(\nabla u)|^{-1}\circ u^{-1}$, and the equivalences follow from the calculation of Section \ref{subsec_elasticity}.
\end{proof}

\subsection{Optimality conditions}

Now we turn to the derivation of the optimality condition. To that extent, we will rely on the notion of the subdifferential of $\Phi$ in the sense of convex analysis.

For any $\omega \in C(\bar{D})$, we define the Legendre transform of $\Phi$ as
\begin{equation*}
\Phi^*(\omega) = \sup_{\mu \in \mathcal{M}(\bar{D})} \, \int_D \omega \, \ddr \mu - \Phi(\mu).
\end{equation*}
 We say that $\omega \in C(\bar{D})$ belongs to the subdifferential of $\Phi$ at $\mu$, and we write $\omega \in \dr \Phi(\mu)$, if $\Phi^*(\omega) = \int \omega \, \ddr \mu - \Phi(\mu)$.

\begin{example}
\label{ex:compressible_pressure}
In the case where 
\begin{equation*}
\Phi_h( \mu ) = \int_D \phi_h\left(\frac{ \ddr \mu}{\ddr \Leb_D}\right) \, \ddr\Leb_D
\end{equation*}
with $\phi_h(s) = h(s^{-1})s$ for a smooth and convex $h$, and if $\mu$ has a smooth density w.r.t. $\Leb_D$ with $\frac{d\mu}{d\Leb_D} \geq \delta>0$ on $D$, then $\partial \Phi_h(\mu)$ consists of the single element
\begin{equation*}
\omega = \phi_h' \left( \frac{\ddr \mu}{\ddr \Leb_D} \right).
\end{equation*}
Written in terms of $u$ and $h$, if $u\in C^1(\bar{\Omega},\bar{D})$ smooth and invertible, then the subdifferential $\partial \Phi_h(u\#\Leb_\Omega)$ consists of the single element
$$
	\omega =\big(-h'(\det \nabla u)\det(\nabla u)+h(\det \nabla u)\big)\circ u^{-1}.
$$
\end{example}

\begin{example}
\label{ex:incompressible_pressure}
As another important case, let $\Phi_\text{inc}$ be defined by 
\begin{equation*}
\Phi_\text{inc}(\mu) = \begin{cases}
0 & \text{if } \mu = \Leb_D \\
+ \infty & \text{otherwise},
\end{cases}
\end{equation*}
then $\Phi^*_\text{inc}(\omega) = \int_D \omega\, \ddr \Leb_D$ and any $\omega \in C(\bar{D})$ belongs to $\dr \Phi_\text{inc}(\Leb_D)$. 
\end{example}

Now we claim that the strong form of the Euler Lagrange equations may be expressed as
\begin{equation}
\label{eq_Euler_Lagrange}
\begin{cases}
\nabla \cdot DW(\nabla u) -D F(u) = (\nabla \omega) \circ u \\
\omega \in \dr \Phi(u \# \Leb_\Omega),
\end{cases}
\end{equation} 
where $DW$ and $DF$ represent the differential with respect to their second argument, that is the one in $\R^{dk}$ and $\R^k$ respectively. The function $\omega \in \dr \Phi(u \# \Leb_\Omega)$ is an unknown: it can be interpreted as the Lagrange multiplier associated to the penalization by $\Phi(u \# \Leb_\Omega)$, or, in short, as a \emph{pressure}. Note that $\dr \Phi(u \# \Leb_\Omega) \neq \emptyset$ implies that $\Phi(u \# \Leb_\Omega) < + \infty$.  

\begin{definition}
Provided that $\nabla \omega$ is well defined, we say that \eqref{eq_Euler_Lagrange} holds weakly if for all $v \in W^{1,p}_0(\Omega, \R^k)$,
\begin{equation*}%\label{eqn:weak-sol}
\int_\Omega\Big( \nabla v : DW(\nabla u)+v\cdot \big(DF(u)+(\nabla \omega)\circ u\big)\Big)\ddr\Leb_\Omega=0.
\end{equation*}
\end{definition}

\noindent Later we will consider the case  $\omega$ is $\lambda$-convex where $\nabla \omega$ will be defined as a measurable selection of the subdifferential of $\omega$.

In the case of $\Phi=\Phi_h$ for $\omega \in \dr \Phi(u \# \Leb_\Omega)$ we have
$$
	(\nabla \omega)\circ u =-\det(\nabla u)\nabla u^{-\top}\nabla h'(\det \nabla u),
$$
which agrees with what one obtains through the direct variation of $u$ with integration by parts and the identity that $\nabla \cdot \det(\nabla u)\nabla u^{-\top}=0$.

\begin{remark}
In elasticity theory the equilibirum equations are not usually written in this way. Indeed if we introduce $ p = \omega \circ u$ and $S = DW(\nabla u)$ the second Piola-Kirchoff stress tensor then \eqref{eq_Euler_Lagrange} can be written 
\begin{equation*}
\nabla \cdot S -DF(u) =  \nabla u^{- \top} \nabla p \ \text{on}\ \Omega.
\end{equation*}
Alternatively, we introduce the Cauchy stress $T=S\, (\nabla u)^\top \circ u^{-1}$, in which case the equilibrium equations are simply
\begin{equation}
\label{eq_Euler_Lagrange_Cauchy}
	\nabla \cdot T - DF(u)\circ u^{-1}=\nabla \omega \ \text{on}\ D.
\end{equation}
The assumption of our main theorem will be about $\omega$, that is about \emph{the pressure in deformed configuration}. 
\end{remark}

Let us justify that \eqref{eq_Euler_Lagrange} are indeed the Euler Lagrange equations in the case where the solution is smooth. Indeed, the existence of a pressure is guaranteed when $u$ is sufficiently smooth, following the arguments of \cite{letallec1981existence}.  Results can also be attained for small forces as in \cite{le1985constitutive}.  In some cases this argument can be weakened to obtain a distributional solution; see the arguments of \cite{giaquinta1994weak}.

By common abuse of notation, in the proposition below and its proof we identify a measure on $D$ with its density w.r.t. $\Leb_D$. Moreover, we will say that $\Phi : \M_+(\bar{D}) \to \R \cup \{ + \infty \}$ is \emph{regular} if it is a convex, bounded from below, lower semi continuous function which coincides with its lower semi continuous envelope of when restricted to measures with a smooth density. Specifically, for every $\mu \in \M_+(\bar{D})$ with $\Phi(\mu) < + \infty$, we assume that there exists a sequence $(\mu_n)_{n \in \N}$ of measures in $W^{j,r}(D)$ with $jr > d$ such that $\Phi(\mu_n)$ converges to $\Phi(\mu)$ when $n \to + \infty$. If $\Phi$ is either $\Phi_{\text{inc}}$ or $\Phi_h$ for a smooth and convex $h$ then it is clearly regular.

\begin{prop}
Assume that $\Phi$ is regular. We suppose that $\dr \Omega, \dr D, g, W, F$ are smooth, and $u\in W_g^{1,p}(\Omega,D)$ is a local minimizer of $E(u)$ such that $0< u\# \Leb_\Omega \in W^{j+1,r}(D)$, the Cauchy stress satisfies $T\in W^{j,r}(D,\R^{dk})$, $DF(u)\circ u^{-1}\in W^{j-1,r}(D,\R^{k})$ for $j\in \mathbb{N}$, $1<r<+\infty$, and $rj>d$.  Then there exists $\omega\in W^{j,r}(D) \cap \dr \Phi(u\#\Leb_\Omega)$ such that \eqref{eq_Euler_Lagrange} holds weakly and \eqref{eq_Euler_Lagrange_Cauchy} holds strongly on $D$.
\end{prop}
\begin{proof}
Let $\mu = u \# \Leb_\Omega$, in particular $\mu \in W^{j+1,r}(D)$.

We start by proving the existence of the pressure $\omega$. We consider variations of the form $u + v \circ u$ for $v:D\rightarrow \R^k$ with $v=0$ on $\partial D$. We define $G : v \mapsto (u+ v \circ u ) \# \Leb_D$. We let 
	$$
		\tilde{E}(v)=\int_\Omega \big( W(\nabla (u+v\circ u)) + F(u + v \circ u) \big) \,  \ddr \Leb_\Omega.
	$$
	The problem to minimize $\tilde{E}(v)$ subject to $G(v)=\mu$ has a local minimum at $v=0$.
	Linearizing the constraint, we have for $z\in C^1(\bar{D})$,
	$$
		\langle z,DG(v)q\rangle =\frac{\ddr}{\ddr t}\int_D z\, \ddr(u+t\, q\circ u)\# \Leb_\Omega\Big|_{t=0}= \frac{d}{dt}\int_D z\circ (\cdot+t\, q)\, \ddr \mu \Big|_{t=0}=\int_D \nabla z\cdot q\, \ddr \mu.
	$$
	The map $DG(v):W^{j+1,r}_0(D)\rightarrow W^{j,r}(D)$, which can be expressed as
	$$
		DG(v)q = -\nabla \cdot \left(q\, \mu\right),
	$$ 
	is continuous, making $G$ $C^1$, and has closed range as a composition of the divergence operator and multiplication by a positive function in $W^{j+1,r}(D)$.  The range of $DG(v)$ consists of functions in $W^{j+1,r}(D)$ that integrate to zero.

	Thus by the Lagrange multiplier theorem there exists $\omega\in W^{-j,r}(D)$ with $\int_D \omega\, \ddr \mu=0$ such that
	$$
		0=D\tilde{E}(0)q + \langle \omega,DG(0)q\rangle
		%=  \int_D \left(-\nabla \cdot T+DF(u)\circ u^{-1} \right) \cdot q dx + \langle DG(v)^\top \omega,q\rangle_{W^{j+1,q}(D,\R^k)} 
	$$
	for all $q\in W^{j+1,r}_0(D,\R^k)$; see \cite[\S 4.14 Proposition 1]{zeidler1995applied}  with $f(v)=\tilde{E}( v)$, $G(v)=G(v)$, $X = W^{j+1,r}_0(D,\R^k)$ and $Y=\{ z\in W^{j,r}(D);\ \int_D z\, d{\Leb}_D = 0\}$. 

	As $ D\tilde{E}(0)q= \int_D \left(-\nabla \cdot T+DF(u)\circ u^{-1} \right) \cdot q \, \ddr \mu$ and $\langle \omega,DG(0)q\rangle=\int \nabla \omega \cdot q \, \ddr \mu$ in the distribution sense, it then follows from $\mu >0$ on $D$ that $\nabla \omega  \in W^{j-1,r}(D,\R^k)$ % $DG(v)^\top \omega \in W^{j-1,r}(D,\R^k)$ and
	 thus $\omega \in W^{j,r}(D)\subset C(\bar{D})$ and \eqref{eq_Euler_Lagrange_Cauchy} holds. Changing variables back to $\Omega$ we have that \eqref{eq_Euler_Lagrange} holds, at least weakly, except it remains to check that $\omega\in \dr \Phi(\mu)$.

	We now suppose that $\omega$ is \emph{not} in the subdifferential of $\dr \Phi(\mu)$. Therefore there is $\nu\in \mathcal{M}(\bar{D})$ and $a>0$ such that
\begin{equation*}
\int_D \omega \, \ddr \nu - \Phi(\nu) \geqslant a + \int_D \omega \, \ddr \mu - \Phi(\mu) = a - \Phi(\mu). 
\end{equation*}	
By regularity of $\Phi$ and continuity of $\omega$, we can take $\nu$ in $W^{j,r}(D)$.
By convexity of $\Phi$ we deduce that for all $t\in [0,1]$,
	$$
\Phi((1-t) \mu + t \nu) \leqslant - at + \Phi(\mu) + t \int_D \omega \, \ddr \nu.  		
	$$
By the implicit function theorem and previous linearization argument (see \cite{zeidler1995applied} \S 4.8 Theorem 4.E with $F(w,v)=G(v)-w$, $Y = W^{j+1,r}_0(D,\R^k)$ and $X = Z =\{ z\in W^{j,r}(D);\ \int_D z\, d{\Leb}_D = 0\}$), we find $q_t\in W^{j+1,r}(D,\R^k)$ for sufficiently small $t$ such that $(u+q_t\circ u)\#\Leb_\Omega=t\nu+(1-t)\mu$ and $DG(0) \dot{q}_0 = \nu - \mu$ where $\dot{q}_0$ is the temporal derivative of $q$ evaluated at $t= 0$. We can write
\begin{equation*}
E(u + q_t \circ u)  = \tilde{E}(q_t) + \Phi((1-t) \mu + t \nu) 
 \leqslant \tilde{E}(q_t) - at + \Phi(\mu) + t \int_D \omega \, \ddr \nu, 
\end{equation*}
and there is equality if $t= 0$. The derivative of the right hand side at $t= 0$ is 
\begin{multline*}
\left.  \frac{\ddr}{\ddr t} \left( \tilde{E}(q_t) - at + \Phi(\mu) + t \int_D \omega \, \ddr \nu \right) \right|_{t= 0} = - a + D\tilde{E}(0) \dot{q}_0 + \int_D \omega \, \ddr \nu \\ 
= -a - \langle \omega,DG(0)\dot{q}_0 \rangle_{W^{j,q}(D)} + \int_D \omega \, \ddr \nu = - a - \int_D \omega \, \ddr (\nu - \mu) + \int_D \omega \, \ddr \nu  = - a.  
\end{multline*}	
Hence by taking $t$ small enough we see that $E(u + q_t \circ u) < E(u)$, which contradicts the local optimality of $u$.	
\end{proof}

\section{Global optimality}

Now let us turn to conditions guaranteeing that a solution of \eqref{eq_Euler_Lagrange} is a global minimizer of the problem. We start with an easy observation. 

\begin{prop}
\label{proposition_main}
Let $u \in W^{1,p}_{g}(\Omega, D)$ with $\Phi(u\#\Leb_\Omega)<+\infty$ and assume that there exists $\omega \in \dr \Phi(u \# \Leb_\Omega)$ such that $u$ is a (unique) global minimizer of 
\begin{equation}
\label{eq_auxiliary_problem}
v \mapsto \int_\Omega \Big(W(\nabla v) +F(v) + \omega(v)\Big)\ddr \Leb_\Omega
\end{equation}
over $W^{1,p}_{g}(\Omega,D)$. Then $u$ is a (unique) minimizer of the energy $E$ introduced in Definition \ref{def_energy}.
\end{prop}

\begin{proof}
Indeed, we can write by definition of $\dr \Phi(u \# \Leb_\Omega)$ that for any competitor $v$,
\begin{align}
E(v) & = \int_\Omega\Big( W(\nabla v) +F(v) \Big)\ddr \Leb_\Omega+ \Phi(v \# \Leb_\Omega)\nonumber \\
& \geqslant \int_\Omega\Big( W(\nabla v) +F(v)\Big)\ddr\Leb_\Omega + \int_D \omega \, \ddr (v \# \Leb_\Omega) - \Phi^*(\omega) \nonumber\\
& = \int_\Omega \Big(W(\nabla v) +F(v)+\omega(v)\Big)\ddr \Leb_\Omega  - \Phi^*(\omega) \nonumber\\
& \geqslant \int_\Omega \Big(W(\nabla u) +F(u)+\omega(u)\Big)\ddr \Leb_\Omega  - \Phi^*(\omega) \label{eqn:second_inequality}\\
& = \int_\Omega\Big( W(\nabla u) +F(u) \Big)\ddr \Leb_\Omega+ \Phi(u \# \Leb_\Omega) =E(u) \label{eqn:last_equality}
\end{align}
where the second inequality (\ref{eqn:second_inequality}) is the assumption on $u$ and the last equality (\ref{eqn:last_equality}) comes from $\omega \in \dr \Phi(u \# \Leb_\Omega)$. 
\end{proof}

\noindent Then we just notice that \eqref{eq_Euler_Lagrange} are also the Euler-Lagrange equations for \eqref{eq_auxiliary_problem}. Hence justifying global optimality for \eqref{eq_auxiliary_problem} is enough to yield global optimality of our original problem. The first result is that convexity of $\omega$ implies global optimality of an equilibrium $u$.

\begin{thm}
\label{theo_general}
Let $u \in W^{1,p}_{g}(\Omega, D)$ and assume that there exists $\omega \in \dr \Phi(u \# \Leb_\Omega)$ such that $\omega$ can be extended to a convex function on $\R^k$ and \eqref{eq_Euler_Lagrange} holds weakly (where $\nabla \omega$ can be any measurable selection of the subdifferential of $\omega$). Then $u$ is a global minimizer of the energy $E$ defined in \eqref{def_energy}. Moreover, if $\omega$, $F$, or $W$ is strictly convex, then $u$ is the unique global minimizer of $E$.
\end{thm}

\begin{proof}
We just use Proposition \ref{proposition_main} by noticing that \eqref{eq_auxiliary_problem} is a convex (respectively, strictly convex) problem provided that $\omega$ is convex (respectively,  one of $\omega$, $F$, or $W$ is strictly convex), and that \eqref{eq_Euler_Lagrange} are the Euler Lagrange equation for \eqref{eq_auxiliary_problem}. Indeed, a solution of the Euler-Lagrange equations of a convex (respectively, strictly convex) problem is a global minimizer (respectively, the unique global minimizer). 
\end{proof}

\begin{remark}
In the case where $W(C)=\frac{1}{2}|C|^2$ is quadratic, $\Phi = \Phi_\text{inc}$ while $F = 0$ then for small $u$ the linearized version of our Problem \eqref{eq_informal} are nothing else than the Stokes equations which read
\begin{equation*}
\begin{cases}
\Delta u = \nabla p & \text{in } \Omega \\
\nabla \cdot u = 0 & \text{in } \Omega \\
u = g & \text{on } \dr \Omega,
\end{cases}
\end{equation*}
and to make the link with our notation we would take $p = \omega$. As $p$ is harmonic if it is not constant then it cannot be convex. In particular, Theorem \ref{theo_general} will not apply in a situation without exterior forces, at least for the linearized problem. 
\end{remark}

If we restrict to uniformly convex energies, we can relax the convexity assumption on the pressure. We recall that a function $f : \R^k \to \R$ is $\lambda$ convex if and only if $y \mapsto f(y) + \frac{\lambda}{2} |y|^2$ is a convex function.

\begin{thm}
\label{theo_dirichlet}
Assume that $W$ is $\lambda_W$ convex and $F$ is $\lambda_F$ convex. Let $\lambda_1(\Omega) > 0$ the first eigenvalue of the Dirichlet Laplacian on $\Omega$.

Let $u \in W^{1,p}_g(\Omega, D)$ for $p\geq 2$ and assume that there exists $\omega \in \dr \Phi(u \# \Leb_\Omega)$ such that $\omega$ can be extended in a $\lambda$-convex function on $\R^k$ with $\lambda \geqslant - \lambda_W \lambda_1(\Omega)-\lambda_F$ and \eqref{eq_Euler_Lagrange} holds weakly (where $\nabla \omega$ can be any measurable selection of the subdifferential of $\omega$).
Then $u$ is the unique minimizer of the energy $E$ introduced in Definition \ref{def_energy}. Moreover, if $\lambda > - \lambda_W \lambda_1(\Omega)-\lambda_F$, $u$ is the unique minimizer of $E$.
\end{thm}

\begin{proof}
Thanks to Proposition \ref{proposition_main}, we just need to study the problem \eqref{eq_auxiliary_problem}. We know that the function $W$ is $\lambda_W$-convex. Combining this with the definition of $\lambda_1(\Omega)$, it is clear that for any $v$ belonging to $W^{1,p}_g(\Omega,D)$ (in particular $u-v$ vanishes on $\dr \Omega$)
\begin{align*}
\int_\Omega \Big( W( \nabla v) & +F(v) \Big)  - \int_\Omega \Big( W( \nabla u)-F(u) \Big)\ddr \Leb_\Omega\\
\geqslant&\  \int_\Omega \Big(-\nabla \cdot (DW(\nabla u)) (v-u) +DF(u)\cdot (v-u)+\frac{\lambda_F}{2}|v-u|^2+ \frac{\lambda_W}{2} |\nabla v - \nabla u|^2\Big)\ddr \Leb_\Omega \\
\geqslant&\  \int_\Omega \Big(-\nabla \cdot (DW(\nabla u)) (v-u) +DF(u)\cdot (v-u)+ \frac{\lambda_F+\lambda_W \lambda_1(\Omega)}{2} |v - u|^2\Big)\ddr \Leb_\Omega.  
\end{align*} 
Hence by adding the functional $\int \omega(v) d\Leb_\Omega$ we see that the problem \eqref{eq_auxiliary_problem} is $(\lambda_F+\lambda_W \lambda_1(\Omega) + \lambda)$ convex, which yields global optimality (if $\lambda \geqslant - \lambda_W \lambda_1(\Omega) - \lambda_F$) and uniqueness (if $\lambda > - \lambda_W \lambda_1(\Omega) - \lambda_F$) for solutions of the Euler-Lagrange equations \eqref{eq_Euler_Lagrange}.
\end{proof}

As a corollary, we deduce a local optimality result. 

\begin{cor}
\label{cor:local_optimality}
Assume that $W$ is $\lambda_W$-convex with $\lambda_W > 0$ and that $\Phi$ is either $\Phi_{\text{inc}}$ or $\Phi_h$. We restrict to the case $\Omega = D$. 

Let $(u, \omega)$ be a smooth solution of the equilibirum equations \eqref{eq_Euler_Lagrange} with $u$ being one to one. Then there exists $r > 0$ such that for every $\tilde{\Omega}$ smooth connected subset of $\Omega$ of diameter bounded by $r$, the function $\left. u \right|_{\tilde{\Omega}}$ is the global maximizer of \eqref{eq_problem_main} with source domain $\tilde{\Omega}$, target domain $u(\tilde{\Omega})$ and boundary conditions $\left. u \right|_{\dr \tilde{\Omega}}$. 
\end{cor}

\begin{proof}
With the assumption that $\Phi$ is either $\Phiinc$ or an integral function of the density, one can notice that $\left. \omega \right|_{\tilde{\Omega}} \in \dr \Phi( \left. u \right|_{\tilde{\Omega}} \# \Leb_{\tilde{\Omega}}  )$ provided that $\omega \in \dr \Phi(u \# \Leb_\Omega)$. 

Thus, to apply Theorem \ref{theo_dirichlet}, we just need to notice that $\lambda_1(\tilde{\Omega})$ goes to $+ \infty$ when the diameter of $\tilde{\Omega}$ goes to $0$, and that every smooth function is locally $\lambda$-convex for some finite $\lambda$.   
\end{proof}

Corollary \ref{cor:local_optimality} highlights that, to prove that a smooth solution of the equilibrium equations is \emph{not} a global minimizer, one cannot rely only on a local perturbation. To build a competitor with smaller energy than a critical point, the global shape of $\Omega$ and/or $D$ must be used.

\section{Examples}
\label{section_examples}

\subsection{Affine deformations}

Assume that $u : \Omega \to D$ is an affine mapping and that $\Phi$ is either $\Phiinc$ or an integral function of the density $\Phi_h$, and $F =0$. Then it is clear that $\dr \Phi(u \# \Leb_\Omega)$ contains a constant function $\omega$. As $u$ clearly satisfies
\begin{equation*}
\nabla \cdot ( DW(\nabla u) ) = (\nabla \omega) \circ u = 0,
\end{equation*}
and that $\omega$ is convex, we can apply Theorem \ref{theo_general}. 

In conclusion, \emph{if $u$ is an affine mapping, then it is a global minimizer of the Problem \eqref{eq_problem_main} with boundary conditions $\left. u \right|_{\dr \Omega}$}. Moreover, if $W$ is strictly convex, it is the unique global minimizer. 

In \cite{knops1986quasiconvexity} a stronger result is proved with an additional assumption on $\Omega$: namely that under the assumption that $\Omega$ is star-shaped and $D$ has the same dimension than $\Omega$, any solution of the equilibrium equations \eqref{eq_Euler_Lagrange} with affine boundary conditions is an affine map. Moreover, the authors can allow for $W$ to be any quasiconvex function.

\subsection{Perturbation of the identity with exterior force}
For simplicity let $W$ be quadratic and  we restrict to $\Phi$ to be $\Phiinc$ the incompressibility constraint. We take $\Omega=D$ and consider the boundary conditions imposed by $g_\epsilon(x)=x+\epsilon g$ for $g\in C^{3,\alpha}(\bar\Omega)$. We take $F(u) = -\nabla \psi \cdot u$ for $\psi\in C^{3,\alpha}(\bar\Omega)$, which is strictly convex.  Then the identity is a global minimizer for $\epsilon=0$ and the pressure is given by $\omega(y) = \psi(y)$.
By the implicit function theorem and analysis of the linearized Stokes equation, 
we conclude that for some $\epsilon_1>0$ there exists a continuous path of solution $u_\epsilon\in C^{3,\alpha}(\Omega,\R^k)$ and $\omega_\epsilon\in C^{2,\alpha}(\Omega,\R^k)$ for $\epsilon\in [0,\epsilon_1)$.  In particular we can select $\epsilon_1>0$ so that $\omega_\epsilon$ remains strictly convex and thus Theorem \ref{theo_general} implies $u_\epsilon$ remains the unique global minimum.  The recent results of \cite{healey2019classical}  show that the local argument can be extended globally by means of a topological degree. The solutions remain a global minimum until $\omega_\epsilon$ loses the convexity imparted by $\psi$. 

\subsection{Pure torsion of a cylinder}

In this example let's take for simplicity $W$ quadratic and (not for simplicity) we restrict to $\Phi$ to be $\Phiinc$ the incompressibility constraint. Let's take $\Omega = D = B(0,1) \times [0,1] \subset \R^3$ a cylinder. A point $x \in \R^3$ will be written $x = (x_h, z) \in \R^2 \times \R$. We will denote by $R_\theta : \R^2 \to \R^2$ the rotation by an angle $\theta$. 

Let $a$ be a parameter. We consider the mapping $u_a : \Omega \to \Omega$ defined by 
\begin{equation*}
u_a(x_h, z) = \begin{pmatrix}
R_{a z} x_h \\
z
\end{pmatrix},
\end{equation*}
that is each horizontal slice is rotated by an angle proportional to $z$. This mapping always satisfies $u \# \Leb_\Omega = \Leb_\Omega$. Moreover, a straightforward computation leads to 
\begin{equation*}
\Delta u_a (x_h, z) = \begin{pmatrix}
- a^2 R_{a z} x_h \\
0 
\end{pmatrix} = (\nabla \omega_a) \circ u_a (x_h,z)
\end{equation*}
provided we define $\omega_a(x_h, z) = - \frac{a^2}{2} |x_h|^2$. In other words, $u_a$ satisfies the equilibrium equations \eqref{eq_Euler_Lagrange}. This is not a surprise, actually $u_a$ satisfies the equilibrium equations for much more general $W$ as it is a well undestood result in elasticity theory, see for instance \cite[Section 3.3]{green1992theoretical}. We make the following observations. 
\begin{itemize}
\item[•] If $a$ is small enough, than $u_a$ is the unique global minimizer of \eqref{eq_problem_main} with boundary conditions $\left. u_a \right|_{\dr \Omega}$. This is a direct consequence of Theorem \ref{theo_dirichlet}. 
\item[•] If $a$ is large enough, than $u_a$ is \emph{not} a global minimizer of the auxiliary problem \eqref{eq_auxiliary_problem} with boundary conditions $\left. g=u_a \right|_{\dr \Omega}$. Indeed, for any displacement $q\in W^{1,2}_0(\Omega ,\R^3)$ with nontrivial horizontal component, $q_h$, the auxiliary energy of $u_a+\epsilon q$ will decrease for sufficiently large $a$ as
\begin{align*}
	&\ \int_\Omega\Big( \frac{1}{2}\big| \nabla u_a+\epsilon \nabla q\big|^2+\omega_a\big(u_a+\epsilon q\big)\Big)\ddr \Leb_\Omega-\int_\Omega\Big( \frac{1}{2}\big| \nabla u_{a}\big|^2+\omega_a\big(u_{a}\big)\Big)\ddr \Leb_\Omega\\
	=&\ \epsilon^2\int_\Omega\Big( \frac{1}{2}\big| \nabla q \big|^2-\frac{a^2}{2}\big|  q_h\big|^2\Big)\ddr \Leb_\Omega,
\end{align*}
where the cancellation of the cross terms occurred as $u_a$ is a solution of the equilibrium equation \eqref{eq_Euler_Lagrange}. For sufficiently small $\epsilon$, if $u$ is bounded then $u_a+\epsilon q\in W^{1,2}_g(\Omega,D)$. 

%  \hugo{Tis should be easy to prove}. 
\item[•] It remains unclear whether there are any other solutions to the equilibrium equations \eqref{eq_Euler_Lagrange} and if $u_a$ is the global minimum of Problem \eqref{eq_auxiliary_problem} with boundary conditions $\left. u_a \right|_{\dr \Omega}$ for all $a$. 
\end{itemize} 

\subsection{Multiple local minima}

Our result provides global minimality but does not prevent the existence of local minimizers or other equilibria. In particular, in \cite{Post1997homotopy} the following example is studied. 

The authors take $\Omega = D = \{ x \in \R^2 \ : \ 0<r_1 \leqslant |x| \leqslant r_2 \} \subset \R^2$ an annulus and as boundary conditions $g$ they take the identity. The energy $\Phi$ is $\Phiinc$ enforcing the incompressibility condition. The identity map is the unique global minimizer of \eqref{eq_problem_main}. However, if $\mathcal{S}_N \subset W^{1,p}_g(\Omega,D)$ is the set of mappings $u \in W^{1,p}_g(\Omega, D)$ such that a.e. radius is mapped by $u$ to a curve of index $N$, then $\mathcal{S}_N$ is connected for the $W^{1,p}(\Omega,D)$ topology. In particular, the energy $E$ admits a local minimizer on every $\mathcal{S}_N$.   

\section{A convex relaxation}

The proofs of Theorems \ref{theo_general} and \ref{theo_dirichlet} rely on convexity arguments: the problem \eqref{eq_auxiliary_problem} is convex under some assumptions on $W$ and the pressure $\omega$. However, the problem \eqref{eq_auxiliary_problem} depends on $\omega$, that is on the solution. In this section, we present a natural convex relaxation of the original problem which can be formulated in a very general context. Then, under the assumption of $\lambda$ convexity on $\omega$, we show that this convex relaxation is tight.  

Our result is reminiscent of Brenier's works (for instance \cite{Brenier1989, Brenier2018}), where he did something similar: taking a non convex problem, formulating a convex relaxation, and finding assumptions on solutions of the original problem guaranteeing that they are also solutions of the relaxed problem. 

Our relaxation relies on a definition of Dirichlet energy for measure valued mappings proposed by Brenier 20 years ago \cite{Brenier2003} and investigated more recently by the third author \cite{Lavenant2017harmonic}. 

An alternative convex relaxation of the problem has been given in \cite{AwiGangbo2014} that uses a measure on the set $\Omega\times D \times \R^+ \times \R^{dk}$, where the last two arguments correspond to distributions for $\det(\nabla u)$ and $\nabla u$.  The dual problem they derive results in an unknown $k$, that relates to the Legendre transform of $\omega$ in our work.  In the convex formulation we present here, the pressure of the deformed configuration, $\omega$, appears in a manner entirely analogous to the dual Kantorovich potential of optimal transport. Let us also mention \cite{mollenhoff2019} which tackles convex relaxation of polyconvex problems of calculus of variations with the use of currents, allowing more general energies than the present work but without any tightness result.

\subsection{The (primal) convex problem}

The idea is to replace $u : \Omega \to D$ by a transport plan $\pi \in \mathcal{M}_+(\bar\Omega \times \bar{D})$ whose marginals are $\Leb_\Omega$ and $\mu$ respectively. If we do that, the potential term and the penalization on the image measure of $u$ are easily translated: what we have gained is that the measure $\mu = u \# \Leb_\Omega$ is replaced by $\mu = \mathrm{proj}_D \# \pi$ a linear expression in $\pi$. What is less obvious is what to do with the term involving the gradient of $u$. 

To that extent, we extend $J\in \mathcal{M}(\bar{\Omega} \times \bar{D}, \R^{dk})$ which is interpreted as the ``flux'' of the transport plan $\pi$. Namely, we will enforce the ``generalized continuity equation''
\begin{equation}
\label{eq_continuity_informal}
\nabla_\Omega \pi + \nabla_D \cdot J = 0,
\end{equation}
and the stored energy will be replaced by 
\begin{equation}
\label{eq_dirichlet_measure_informal}
\iint_{\Omega \times D} W\left(\frac{\ddr J}{\ddr \pi}\right) \ddr \pi
\end{equation}
which is a jointly convex function of $\pi$ and $J$. We explain below in Lemma \ref{lemma_relaxation} how to embed our original problem into this convex relaxation.

Specifically, we say that a nonnegative measure $\pi \in \M_+(\bar{\Omega} \times \bar{D})$ and a matrix valued measure $J \in \M(\Omega \times D, \R^{dk})$ satisfy the generalized continuity equation with boundary conditions $g$ if and only if for all $\varphi \in C^1(\bar{\Omega} \times \bar{D}, \R^d)$ there holds
\begin{equation}
\label{eq_continuity_rigorous}
\iint_{\Omega \times D} \nabla_\Omega \cdot \varphi \, \ddr \pi + \iint_{\Omega \times D} \nabla_D \varphi : \ddr J = \int_{\dr \Omega} \varphi(g) \cdot \nO \, \ddr \sigma,
\end{equation} 
where $\varphi(g)(x)=\varphi(x,g(x))$, $\nO$ is the outward unit normal and $\sigma$ is the surface area measure of $\Omega$.
This is the weak form of \eqref{eq_continuity_informal} with the boundary conditions being given by $\pi(x, \cdot) = \delta_{g(x)}$ for $x \in \dr \Omega$.

\begin{definition}
We say that a triple $(\pi, J, \mu)$ where $\pi \in \M_+(\bar{\Omega} \times \bar{D})$, $J \in \M(\bar{\Omega} \times \bar{D}, \R^{dk})$ and $\mu \in \M_+(\bar{D})$ is admissible if the marginals of $\pi$ are $\Leb_\Omega$ and $\mu$ respectively and if $(\pi, J)$ satisfy \eqref{eq_continuity_rigorous} the generalized continuity equation with boundary conditions $g$. 

\noindent For any admissible triple $(\pi, J, \mu)$, we define its (relaxed) energy by 
\begin{equation*}
E_r(\pi, J, \mu) = \iint_{\Omega \times D}\left[W\left(\frac{\ddr J}{\ddr \pi}\right)+F\right] \ddr \pi  + \Phi(\mu).
\end{equation*}
\end{definition}

\noindent This energy is convex and the set of admissible triples is a convex set: we have precisely built these objects for that purpose.

\begin{remark}
As proved in \cite{Lavenant2017harmonic}, in the case where $W(C) = \frac{1}{2}|C|^2$ is quadratic this energy admits a metric formulation. 

Let us denote by $\W_2$ the quadratic Wasserstein distance on $\M_+(\bar{D})$, extended to $+ \infty$ if the two measures do not have the same total mass (see for instance \cite{villani2003topics}\cite[Chapter 5]{SantambrogioOTAM} for a definition). We fix $\pi \in \M_+(\bar{\Omega} \times \bar{D})$ whose first marginal is $\Leb_\Omega$ and we denote by $(\pi_x)_{x \in \Omega}$ its disintegration with respect to the first component, in particular $\pi_x \in \M_+(\bar{D})$ for a.e. $x \in \Omega$. If $D$ is convex, \cite[Theorem 3.26]{Lavenant2017harmonic} yields  
\begin{equation*}
\min_{J} \iint_{\Omega \times D} \frac{1}{2}\left|\frac{\ddr J}{\ddr \pi}\right|^2 \ddr \pi = \lim_{\varepsilon \to 0} C_d \iint_{\Omega \times \Omega} \frac{\W_2^2( \pi_x, \pi_{x'} )}{2 \varepsilon^{d+2}} \1_{|x - x'| \leqslant \varepsilon} \, \ddr x \, \ddr x',
\end{equation*} 
where the infimum is taken over all $J \in \M(\bar{\Omega} \times \bar{D}, \R^{dk})$ such that $(\pi, J)$ satisfy \eqref{eq_continuity_rigorous} and $C_d$ is a dimensional constant.  

The right hand side in the equation above can be interpreted as a metric definition of the Dirichlet energy \cite{Korevaar1993} for the mapping $x \mapsto \pi_x$ valued in $\M_+(\bar{D})$ endowed with the distance $\W_2$.  
\end{remark}

The relaxed energy is a convex relaxation of our original energy in the following sense.

\begin{lem}
\label{lemma_relaxation}
Let $u \in W^{1,p}_g(\Omega, D)$ be given. We define $\mu_u = u \# \Leb_\Omega$, and we define $\pi_u$ and $J_u$ by, 
\begin{equation*}
\iint_{\Omega \times D} a \, \ddr \pi_u = \int_\Omega a(x,u(x)) \, \ddr \Leb_\Omega(x) \hspace{1cm} \text{and} \hspace{1cm} \iint_{\Omega \times D} B \, \ddr J_u = \int_\Omega B(x,u(x)) : \nabla u(x) \, \ddr \Leb_\Omega(x),
\end{equation*}
for any test functions $a \in C(\bar{\Omega} \times \bar{D})$ and $B \in C(\bar{\Omega} \times \bar{D}, \R^{dk})$.
Then $(\pi_u, J_u, \mu)$ is admissible and $E_r(\pi_u, J_u, \mu_u) = E(u)$. 
\end{lem}

\noindent We leave the proof as an exercise to the reader, see \cite[Proposition 5.2]{Lavenant2017harmonic} where it is written explicitly. Note that it could be written $\pi_u = \delta_{u} \Leb_\Omega$ and $J_u = \nabla u \, \delta_{u}\Leb_\Omega$. 

\subsection{The dual problem}

Let us write the dual\footnote{As always in optimal transport, we perpetuate the confusion as the dual should be rather called primal: measures are the dual of continuous functions and not the other way around.} of the problem above. It can be guessed by a formal $\inf-\sup$ exchange analogous to what was done in \cite{Brenier2003} and \cite{Lavenant2017harmonic}. The absence of duality gap could be obtained via Fenchel-Rockafellar theorem as written in \cite{Lavenant2017harmonic}, however we will not need it hence we will not prove it. Dual attainment is an open question.

There will be three dual variables: a Lagrange multiplier $\varphi$ for the generalized continuity equation, and then two Lagrange multipliers $\psi, \omega$ for the marginal constraints on $\pi$. 

\begin{definition}
We say that a triple $(\varphi, \psi, \omega)$ where $\varphi \in C^1(\bar{\Omega} \times \bar{D}, \R^d)$, $\psi \in C(\bar{\Omega})$ and $\omega \in C(\bar{D})$ is admissible if for all $x,y \in \Omega \times D$, 
\begin{equation}
\label{eq_dual_constraint}
\psi(x) + \omega(y) +F(x,y) \geqslant \nabla_\Omega \cdot \varphi (x,y) + W^*\big(x,\nabla_D \varphi(x,y)\big). 
\end{equation}

\noindent For any admissible triple $(\varphi, \psi, \omega)$, we define its (relaxed) dual energy by 
\begin{equation*}
E^*_r(\varphi, \psi, \omega) = \int_{\dr \Omega} \varphi(g) \cdot \nO \, \ddr \sigma - \int_\Omega \psi\, \ddr \Leb_\Omega - \Phi^*(\omega). 
\end{equation*}
\end{definition}

\begin{remark}
In the case where there is no integral energy, that is we take $\varphi = 0$, and if $\Phi = \Phiinc$, then our dual problem is exactly Kantorvich dual problem \cite[Chapter 1]{SantambrogioOTAM} for the cost $-F$. 
\end{remark}

\begin{prop}[Weak duality]
\label{prop_weak_duality}
Let $(\pi, J, \mu)$ be admissible for the primal problem and $(\varphi, \psi, \omega)$ admissible for the dual problem. Then 
\begin{equation*}
E_r(\pi, J, \mu) \geqslant E^*_r(\varphi, \psi, \omega)
\end{equation*}
and equality holds if and only if
\begin{equation*}
\begin{cases}
\dst{\nabla_D\varphi(x,y)=DW\left(x,\frac{\ddr J}{\ddr \pi}(x,y)\right) } & \text{for } \pi\text{-a.e. } (x,y) \in \Omega \times D, \\
\text{equality holds in \eqref{eq_dual_constraint}} & \text{for } \pi\text{-a.e. } (x,y) \in \Omega \times D, \\
\omega \in \dr \Phi(\mu). & \\
\end{cases}
\end{equation*} 
\end{prop}

\begin{proof}
Let's compute the difference: we take $(\pi, J, \mu)$ admissible for the primal problem and $(\varphi, \psi, \omega)$ admissible for the dual problem. Then 
\begin{align*}
E_r & (\pi, J, \mu) - E^*_r(\varphi, \psi, \omega) \\
& = \iint_{\Omega \times D}\left[ W\left(\frac{\ddr J}{\ddr \pi}\right)+F\right]\ddr \pi + \Phi(\mu) - \int_{\dr \Omega} \varphi(g) \cdot \nO \, \ddr \sigma + \int_\Omega \psi\, \ddr \Leb_\Omega + \Phi^*(\omega) \\
& \geqslant \iint_{\Omega \times D}\left[ W\left(\frac{\ddr J}{\ddr \pi}\right)+F\right]\ddr \pi   + \int_\Omega \psi\, \ddr \Leb_\Omega + \int_D \omega \, \ddr \mu - \int_{\dr \Omega} \varphi(g) \cdot \nO \, \ddr \sigma \\
& =  \iint_{\Omega \times D}\left[ W\left(\frac{\ddr J}{\ddr \pi}\right)+F+\psi+\omega\right]\ddr \pi  - \int_{\dr \Omega} \varphi(g) \cdot \nO \, \ddr \sigma,
\end{align*} 
where we have used the definition of $\Phi^*$ and then the assumption that the marginals of $\pi$ are $\Leb_\Omega$ and $\mu$. Using the generalized continuity equation and integrating the constraint \eqref{eq_dual_constraint} with respect to $\pi$, 
\begin{align*}
\iint_{\Omega \times D} & \left[ W\left(\frac{\ddr J}{\ddr \pi}\right)+F+\psi+\omega\right]\ddr \pi  - \int_{\dr \Omega} \varphi(g) \cdot \nO \, \ddr \sigma \\
& = \iint_{\Omega \times D}\left[ W\left(\frac{\ddr J}{\ddr \pi}\right)+F+\psi+\omega-\nabla_\Omega \cdot \varphi-\nabla_D\varphi:DW\left(\frac{\ddr J}{\ddr \pi}\right)\right]\ddr \pi \\
& \geqslant \iint_{\Omega \times D}\left[ W\left(\frac{\ddr J}{\ddr \pi}\right)+W^*\big(\nabla_D\varphi\big)-\nabla_D\varphi:DW\left(\frac{\ddr J}{\ddr \pi}\right)\right]\ddr \pi  .
\end{align*}
Eventually, using the definition of $W^*$ we get 
\begin{equation*}
E_r  (\pi, J, \mu) - E^*_r(\varphi, \psi, \omega) \geqslant 0. 
\end{equation*} 
Tracking back all the inequalities, we deduce the necessary and sufficient conditions for the absence of duality gap. 
\end{proof}

\subsection{A particular solution for the dual}

So far, putting together Lemma \ref{lemma_relaxation} and Proposition \ref{prop_weak_duality}, we have
\begin{equation*}
\sup_{(\varphi, \psi, \omega) \text{ admissible}} E^*_r(\varphi, \psi, \omega) \leqslant \min_{(\pi, J, \mu) \text{ admissible}} E_r(\pi, J, \mu) \leqslant \min_{u \in W^{1,p}_g(\Omega, D)} E(u).
\end{equation*}
To prove that everything boils down to a big equality, it is sufficient to provide one competitor in the dual which matches $E(u)$.

We handle this for the case of the Dirichlet energy, for which we can extend the result to $\lambda$ convex pressure as in Theorem \ref{theo_dirichlet}. We recall that $\lambda_1(\Omega) > 0$ is the first eigenvalue of the Dirichlet Laplacian on $\Omega$.

\begin{prop}
\label{prop_no_gap}
Let $W(C)=\frac{1}{2}|C|^2$ and assume $F$ is smooth and $\lambda_F$ convex in its second variable.  We suppose that $u$ is a smooth solution of the Euler-Lagrange equation \eqref{eq_Euler_Lagrange} for some $\omega \in \dr \Phi(u \# \Leb_\Omega)$. Assume $\omega$ is $\lambda$-convex with $\lambda > - \lambda_1(\Omega) - \lambda_F$. Then there exists $(\varphi_u, \psi_u, \omega_u)$ admissible such that 
\begin{equation*}
E^*_r(\varphi_u, \psi_u, \omega_u) = E(u).
\end{equation*}
\end{prop}

\begin{proof}
Recall that if we define $(\pi_u, J_u, \mu_u)$ as in Lemma \ref{lemma_relaxation} we just need to prove that $E^*_r(\varphi_u, \psi_u, \omega_u) = E_r(\pi_u, J_u, \mu_u)$

We take $\omega_u = \omega$. In particular $\omega_u \in \dr \Phi(\mu_u)$.  

We choose $\varepsilon > 0$ such that $\lambda + \lambda_1(\Omega) + \lambda_F > \varepsilon$.
Let $w\in C^1(\bar{\Omega},\R^d)$ be a function such that $\nabla \cdot w + |w|^2 < - \lambda_1(\Omega) + \varepsilon$: see Lemma \ref{lem_change_variables} below for the existence of such a $w$. We define $\varphi_u$, which is valued in $\R^d$ and that we see as a row vector as 
\begin{equation*}
\varphi_u(x,y) = y^\top \nabla u(x) + \frac{|u(x) - y|^2}{2} w(x)^\top. 
\end{equation*} 
Without the quadratric term (that is if $w = 0$) we retrieve the solution proposed by Brenier in \cite{Brenier2003} for the case where there is no constraint on $\mu$, that is $\Phi = 0$. This quadratic term will be crucial to go from $\omega$ convex to $\omega$ that is $\lambda$-convex. Taking the derivative of $\varphi_u$ w.r.t. $y$ and evaluating at $y = u(x)$, we get 
\begin{equation*}
\nabla_D \varphi_u(x,u(x)) = \nabla u(x)
\end{equation*} 
which exactly reads as 
\begin{equation*}
\nabla_D \varphi_u(x,y)=DW\left(\frac{\ddr J_u}{\ddr \pi_u}(x,y)\right) \hspace{0.5cm} \text{for } \pi_u\text{-a.e. } (x,y) \in \Omega \times D.
\end{equation*}

It remains to choose the function $\psi_u$ such that the constraint \eqref{eq_dual_constraint} is satisfied. Let us compute the right hand side of \eqref{eq_dual_constraint}: we find 
\begin{align}
\label{eq_aux_expansion}
\nabla_\Omega \cdot \varphi_u & (x,y)  + \frac{1}{2} |\nabla_D \varphi_u(x,y)|^2 \\
& = y^\top \Delta u(x) + (\nabla \cdot w)(x) \frac{|y - u(x)|^2}{2} - (u(x) - y)^\top \nabla u(x) w(x) + \frac{1}{2} \left| \nabla u(x) + w(x) (u(x) - y)^\top \right|^2 \nonumber \\
& = y^\top \Delta u(x) + \frac{1}{2} |\nabla u(x)|^2 + \frac{1}{2} \left( \nabla \cdot w(x) + |w(x)|^2 \right) |y - u(x)|^2 \nonumber  \\
& \leqslant y^\top \Delta u(x) + \frac{1}{2} |\nabla u(x)|^2 + \frac{-\lambda_1(\Omega) + \varepsilon}{2} | y - u(x)|^2, \nonumber
\end{align}
where the last inequality derives from the way $w$ was chosen and is an equality if $y = u(x)$. Notice that some cancellation of the cross term occurred because of the quadratic structure of the problem. Given this computation, we define
\begin{equation*}
\psi_u(x) = - \omega_u\big(u(x)\big) - F(x,u(x))  + (\nabla \omega_u)\big(u(x)\big) \cdot u(x) + DF(x,u(x)) \cdot u(x) + \frac{1}{2} |\nabla u(x)|^2.
\end{equation*} 
To see if \eqref{eq_dual_constraint} is satisfied we compute 
\begin{align*}
\psi_u (x) & + \omega_u(y)  + F(x,y)  - \nabla_\Omega \cdot \varphi_u  (x,y)  - \frac{1}{2} |\nabla_D \varphi_u(x,y)|^2 \\
& \geqslant \omega_u(y) + F(x,y) - \omega_u(u(x)) - F(x,u(x)) + (\nabla \omega_u)\big(u(x)\big) \cdot u(x) + DF(x,u(x)) \cdot u(x)  \\  
& \hspace{1cm} - y^\top \Delta u (x) + \frac{\lambda_1(\Omega) - \varepsilon}{2}  |y - u(x)|^2 \\
& = \omega_u(y) + F(x,y) - \omega_u\big(u(x)\big) - F(x,u(x)) - \left[ (\nabla \omega_u)\big(u(x)\big) + DF(x,u(x))  \right]\cdot \big(y - u(x)\big) \\
& \hspace{1cm} + \frac{\lambda_1(\Omega) - \varepsilon}{2}|y - u(x)|^2 \\
& \geqslant \frac{\lambda + \lambda_F + \lambda_1(\Omega)- \varepsilon}{2} |y - u(x)|^2.
\end{align*}
In this computation, the equality derives from the equilibrium equations \eqref{eq_Euler_Lagrange} and the last inequality comes from the $(\lambda + \lambda_F)$-convexity of $y \mapsto \omega(y) + F(x, y)$. But $\varepsilon$ was chosen in such a way that $\lambda + \lambda_F + \lambda_1(\Omega)- \varepsilon > 0$ hence $(\varphi_u, \psi_u, \omega_u)$ satisfies \eqref{eq_dual_constraint} and there is equality if and only if $y = u(x)$, that is on the support of $\pi_u$.    
\end{proof}

\begin{remark}
\label{rk_quadratic}
The cancellation that occurred in \eqref{eq_aux_expansion} is crucial to get the results and is what would break down if we replaced the Dirichlet energy by $\int_\Omega W(\nabla u)\ddr \Leb_\Omega$ with $W$ convex.  However, in the case that $\omega$ is convex with general energy, we obtain the same result as Proposition \ref{prop_no_gap} by choosing $\varphi_u(x,y)=y^\top DW\big(\nabla u(x)\big)$.
\end{remark}

\noindent During the proof we have used the following lemma, which relies on a standard change of variables in nonlinear elliptic equations. 

\begin{lem}
\label{lem_change_variables}
Let $\lambda > - \lambda_1(\Omega)$. Then there exists $w \in C^\infty(\bar{\Omega}, \R^d)$ such that 
\begin{equation*}
\nabla \cdot w + |w|^2 < \lambda
\end{equation*}
everywhere on $\Omega$.
\end{lem}

\begin{proof}
We will look rather for $z \in C^\infty(\bar{\Omega})$ such that 
\begin{equation*}
\Delta z + |\nabla z|^2 < \lambda
\end{equation*}
as we can always take $w = \nabla z$. To that extent, let $f$ a smooth strictly positive function such that $\Delta f \leqslant  \lambda f$.
This is always possible by mollifying the first eigenvalue of the Dirichlet Laplacian on $\Omega$. Then defining $z = \ln f$ works as we leave the reader to check.  
\end{proof}

\subsection{Another proof of Theorem \ref{theo_dirichlet}}

%With that
We can give another proof of Theorem \ref{theo_dirichlet} for the Dirichlet energy, that is, when $W(C)= \frac{1}{2} C^2$, relying on this convex relaxation. 

\begin{proof}[Proof of Theorem \ref{theo_dirichlet}]
Let $u$ be a smooth solution of the Euler-Lagrange equation \eqref{eq_Euler_Lagrange} for some $\omega \in \dr \Phi(u \# \Leb_\Omega)$. We assume $\omega$ is $\lambda$-convex with $\lambda > -\lambda_1(\Omega) - \lambda_F$. 

Let $(\varphi_u, \psi_u, \omega_u)$ be the optimal solution of the dual problem built in the proof of Proposition \ref{prop_no_gap}. For any competitor $v \in W^{1,2}_g(\Omega, D)$ there holds
\begin{equation*}
E(v) \geqslant E_r(\pi_v, J_v, \mu_v) \geqslant E_r^*(\varphi_u, \psi_u, \omega_u) = E(u)
\end{equation*}
where we have used successively Lemma \ref{lemma_relaxation}, Proposition \ref{prop_weak_duality} and then Proposition \ref{prop_no_gap}. Moreover, if there is equality then by Proposition \ref{prop_weak_duality} the constraint \eqref{eq_dual_constraint} (with $(\varphi_u, \psi_u, \omega_u)$) must be an equality on the support of $\pi_v$, that is for $y = v(x)$. However we have seen in the proof of Proposition \ref{prop_weak_duality} that equality happens only for $y = u(x)$. This yields $u = v$.
\end{proof}

\section*{Acknowledgments}

The authors are partially supported by the Natural Sciences and Engineering Research Council of Canada (NSERC) through the Discovery Grant program. HL is partially supported by the Pacific Institute for the Mathematical Sciences (PIMS) through a PIMS postdoctoral fellowship.

\bibliography{Elasticity}
\bibliographystyle{plain}

\end{document}